\pdfoutput=1
\documentclass{amsart}
\usepackage{graphicx}
\theoremstyle{definition}
\newtheorem{definition}{Definition}

\newtheorem{problem}{Problem}
\newtheorem{fact}{Fact}
\newtheorem{theorem}{Theorem}

\newcommand{\rI}{\operatorname{RI}}
\newcommand{\rII}{\operatorname{RI\!I}}
\newcommand{\rIII}{\operatorname{RI\!I\!I}}
\newcommand{\srII}{\operatorname{strong~RI\!I}}
\newcommand{\wrII}{\operatorname{weak~RI\!I}}
\newcommand{\srIII}{\operatorname{strong~RI\!I\!I}}
\newcommand{\wrIII}{\operatorname{weak~RI\!I\!I}}
\newcommand{\odd}{\operatorname{odd}}
\newcommand{\coh}{\operatorname{Coh}}
\begin{document}
\author{Noboru Ito}
\thanks{The work of N.~Ito was partially supported by a Waseda University Grant for Special Research Projects (Project number: 2015K-342) and Japanese-German Graduate Externship.  N.~Ito was a project researcher of Grant-in-Aid for Scientific Research (S) 24224002 (April 2016 -- March 2017).}
\address{Graduate School of Mathematical Sciences, The University of Tokyo, 3-8-1, Komaba, Meguro-ku, Tokyo, 153-8914, Japan}
\email{noboru@ms.u-tokyo.ac.jp}
\author{Yusuke Takimura}
\address{Gakushuin Boys' Junior High School, 1-5-1, Mejiro, Toshima-ku Tokyo, 171-0031, Japan}
\email{Yusuke.Takimura@gakushuin.ac.jp}
\title{Thirty-two equivalence relations on knot projections}
\begin{abstract}
We consider 32 homotopy classifications of knot projections (images of generic immersions from a circle into a 2-sphere).  These 32 equivalence relations are obtained based on which moves are forbidden among the five type of Reidemeister moves.  We show that 32 cases contain 20 non-trivial cases that are mutually different.  To complete the proof, we obtain new tools, i.e., new invariants.  
\end{abstract}
\maketitle
\section{Introduction}
A {\it{knot projection}} is the image of a generic immersion from a circle into a $2$-sphere.  In particular, every self-intersection is a transverse double point, which is simply called a {\it{double point}}.  Every double point consists of two branches and thus,  if the two branches are given over/under information for every double point of a knot projection, we can then obtain a {\it{knot diagram}}.  

Several interesting homotopy classes have been considered by restricting the {\it{Reidemeister moves}} that consist of three types of local replacements of knot projections, as shown in Fig.~\ref{32f001} \cite{arnold, HY, IT_12, ITT, IT_w123, IT_circle, khovanov, polyak} (for other works, see \cite{ito}).  
\begin{figure}[htbp]
\begin{center}
\includegraphics[width=10cm]{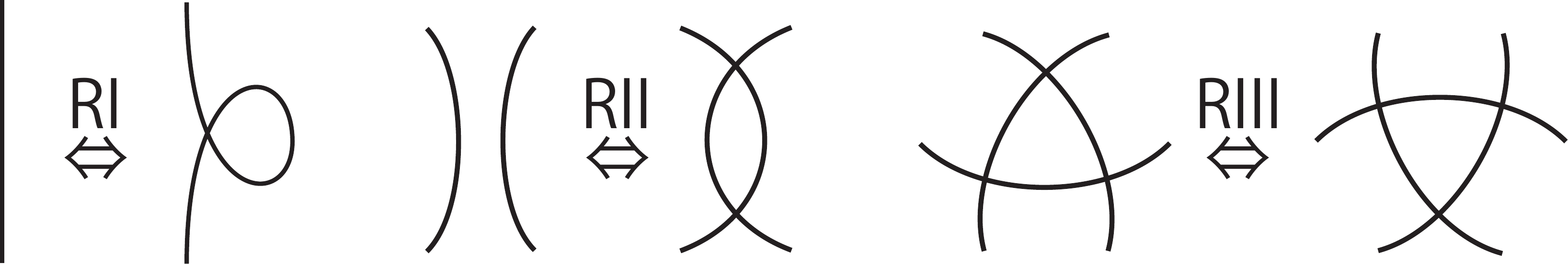}
\caption{Reidemeister moves}
\label{32f001}
\end{center}
\end{figure}
Because a knot projection is a single component, we consider five types of Reidemeister moves, namely, $\rI$, $\srII$, $\wrII$, $\srIII$, and $\wrIII$, which are the local replacements defined in Fig.~\ref{32f002}.  
\begin{figure}[htbp]
\begin{center}
\includegraphics[width=10cm]{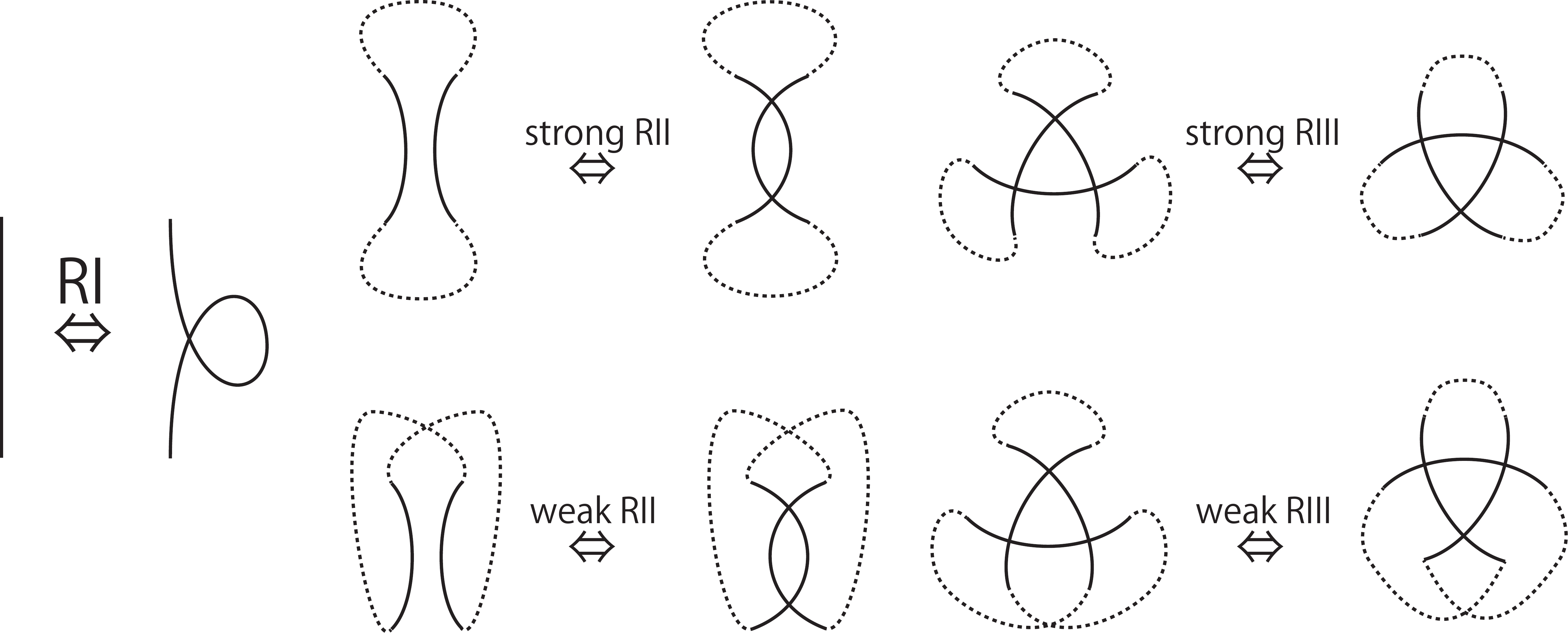}
\caption{RI, strong~RI\!I, weak~RI\!I, strong~RI\!I\!I, weak~RI\!I\!I.  Each dotted arc indicates the connection between branches.}
\label{32f002}
\end{center}
\end{figure}

We can consider 32 ($= 2^5$) equivalence classes by restricting the Reidemeister moves.  Naturally, we have Problem~\ref{p1} as below.  
\begin{problem}\label{p1}
\noindent(1) Which equivalence classes of knot projections are non-trivial?

\noindent(2) Which two equivalence relations on knot projections are independent?
\end{problem}
\noindent Theorem~\ref{main1} described in this paper solves Problem~\ref{p1}.  
\section{Main Result}
\begin{definition}\label{dfn1}
Let $\mathcal{R}$ $=$ $\{ \rI, \srII, \wrII, \srIII, \wrIII \}$ and let $\mathcal{S}$ be a subset of $\mathcal{R}$.  We say that two knot projections $P$ and $P'$ are $\mathcal{S}$-{\it{equivalent}} if they can be related by a finite sequence consisting of the elements of $\mathcal{S}$.  We denote this equivalence by $\sim_{\mathcal{S}}$.  
There are 32 ($= 2^5$) possibilities of type 
\[  \{ {\text{knot projections}} \} /\sim_{\mathcal{S}} \]
that is denoted by $\mathcal{C}_{\mathcal{S}}$.  
\end{definition}
\begin{theorem}\label{main1}
Let $\mathcal{C}_{\mathcal{S}}$ be as defined in Definition~\ref{dfn1}.   
Among the 32 possible sets, 8 sets are trivial (i.e.,~all knot projections are equivalent to the trivial knot projection) and the remaining 24 sets are equivalent to the following 20 sets that are non-trivial and mutually different: 

\noindent $\mathcal{S}$ $=$ $\{ \srII, \wrII, \srIII, \wrIII \}$, $\{ \srII, \srIII, \wrIII \}$, $\{ \wrII, \srIII \}$, $\{ \wrII, \srIII, \wrIII \}$, $\{ \srII \}$, $\{ \wrII \}$, $\{ \srII, \wrII \}$, $\{ \wrII, \wrIII \}$, $\{ \srIII, \wrIII \}$, $\{ \srIII \}$, $\{ \wrIII \}$, $\{ \rI \}$, $\{ \rI, \srII \}$, $\{ \rI, \wrII \}$, $\{ \rI, \srII, \wrII \}$, $\{ \rI,$ $\srIII \}$, $\{ \rI, \wrIII \}$, $\{ \rI, \srIII, \wrIII \}$, $\{ \rI, \wrII, \wrIII \}$, or $\emptyset$.   
\end{theorem}
\section{Invariants of knot projections}
\subsection{New tools---new invariants}
In this section, we introduce a new invariant $\coh^{\odd}(P)$ for a knot projection $P$ under $\wrII$ and $\srIII$ to detect one of the two cases: $\mathcal{C}_{\{ \wrII,~\srIII \}}$ and $\mathcal{C}_{\{ \wrII,~\srIII,~\wrIII \}}$.  

Let $P$ be a knot projection with an arbitrary orientation.  
If the orientation induces an orientation of an $n$-gon of $P$ (i.e., the orientations of $n$ edges are coherent), the $n$-gon is called a {\it{coherent}} $n$-gon.  An $n$-gon that is not coherent is called an {\it{incoherent}} $n$-gon.  The sum of the number of coherent $(2m+1)$-gons ($m \in \mathbb{Z}_{\ge 0}$) is called the {\it{odd coherent number}}.  We set the function $\coh^{\odd}$ from the set of knot projections to $\{0, 1\}$ such that
\[
\coh^{\odd}(P) = 
\begin{cases}
0~~{\text{if odd coherent number is}}~0, \\
1~~{\text{if odd coherent number is not}}~0.
\end{cases}  
\]
By definition, $\coh^{\odd}(P)$ does not depend on the choice of the orientation of $P$.  
\begin{theorem}\label{thm2}
Let $P$ be a knot projection.  Then, $\coh^{\odd}(P)$ is invariant under $\wrII$ and $\srIII$.  
\end{theorem}
\begin{proof}
\noindent $\bullet$ (Invariance under $\srIII$) 
A single $\srIII$ between two knot projections $P$ and $P'$ preserves the condition $\coh^{\odd}(P)$ $=$ $1$ $=$ $\coh^{\odd}(P')$ (see Fig.~\ref{32f002}).

\noindent $\bullet$ (Invariance under $\wrII$)
\begin{figure}[htbp]
\begin{center}
\includegraphics[width=7cm]{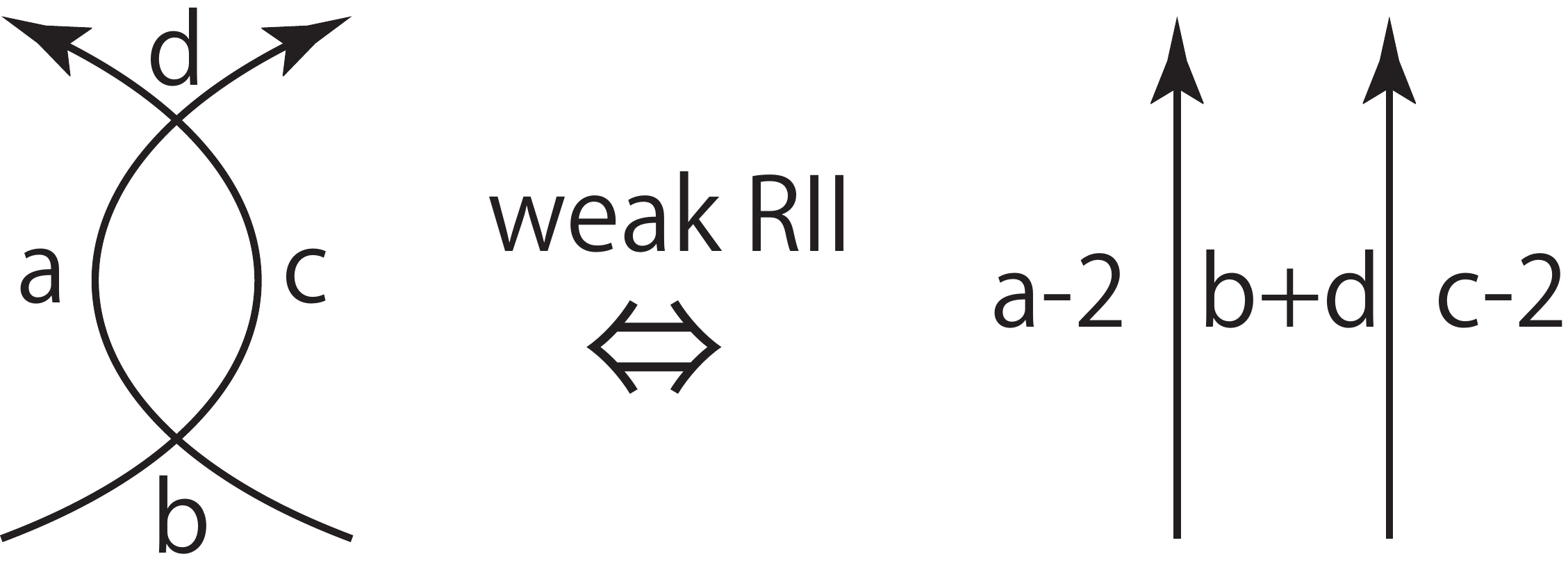}
\caption{A single $\wrII$ and the sides of $a$-, $b$-, $c$-, and $d$-gons}
\label{32f16}
\end{center}
\end{figure}
Suppose that $a (\ge 2)$, $b (\ge 1)$, $c (\ge 2)$, and $d (\ge 1)$ are integers.  The left (resp.~right) part of Fig.~\ref{32f16} represents the local situation of a knot projection $P$ (resp.~$P'$).  Assume that $\coh^{\odd}(P)=0$.  Then, $a$-gon (resp.~$c$-gon) is an incoherent $a$-gon (resp.~$c$-gon) or coherent even $a$-gon (resp.~$c$-gon), i.e., a coherent $a$-gon (resp.~$c$-gon) where $a$ (resp.~$c$) is an even integer.  Any other polygon of $P'$ is an incoherent or coherent even $n$-gon
 for some $n$.  This is because $a$ $\equiv$ $a-2$ (mod $2$) and $c$ $\equiv$ $c-2$ (mod $2$).  
 If the $(b+d)$-gon in $P'$ is a coherent $(b+d)$-gon, which is a boundary of a disk, we cannot apply a single $\wrII$ to obtain $P$; this indicates that there is a contradiction.  Thus, $(b+d)$-gon is an incoherent $(b+d)$-gon.  Then, $\coh^{\odd}(P')=0$.  Similarly, if $\coh^{\odd}(P')=0$, it can be easily seen that $\coh^{\odd}(P)=0$.  As a result, $\coh^{\odd}(P)$ $=$ $0$ $\Leftrightarrow$ $\coh^{\odd}(P')=0$ and then, $\coh^{\odd}(P)$ $=$ $\coh^{\odd}(P')$.  
\end{proof}
Next, we introduce an invariant $C(P)$ in a similar fashion.  
For a knot projection $P$, $c(P)$ denotes the number of double points.  We define a function from the set of all knot projections to $\{0, 1\}$ such that
\[
C(P)=
\begin{cases}
0~~~{\text{if}}~c(P) = 0,\\
1~~~{\text{if}}~c(P) \neq 0.
\end{cases}
\]
\begin{theorem}\label{thm3}
$C(P)$ is invariant under $\wrII$, $\wrIII$, and $\srIII$.  
\end{theorem}
\begin{proof}
A single $\wrII$, $\wrIII$, or $\srIII$ between two knot projections $P$ and $P'$ preserves the condition $C(P)$ $=$ $1$ $=$ $C(P')$, as can be seen in Fig.~\ref{32f002}.  
\end{proof}
\begin{theorem}\label{thm_s2w2}
Let $P$ be a knot projection.  A knot projection $P^{2sr}$ (resp.~$P^{2wr}$) with no coherent (resp.~incoherent) $2$-gons is obtained from $P$ only decreasing double points by a finite sequence consisting of $\srII$ (resp.~$\wrII$).  Two knot projections, $P^{2sr}$ (resp.~$P^{2wr}$) and ${P'}^{2sr}$ (resp.~${P'}^{2wr}$), are sphere isotopic if and only if $P$ and $P'$ are related by a finite sequence consisting of $\srII$ (resp.~$\wrII$).  
\end{theorem}
\begin{proof}
We can prove the claim by considering the general $\rII$ restricted to only $\srII$ (resp.~$\wrII$) in \cite[Proof of Theorem~2.2 (2)]{IT_12}.  
\end{proof}
\subsection{Known invariants and facts}
Let $P$ be a knot projection.  For every double point of $P$, we provide over/under information, as shown in Fig.~\ref{32f0_1}; as a result, we obtain a knot diagram $K_P$, with over/under information for every double point (cf.~\cite[Sec.~6.6]{polyak}).  
\begin{figure}[htbp]
\begin{center}
\includegraphics[width=5cm]{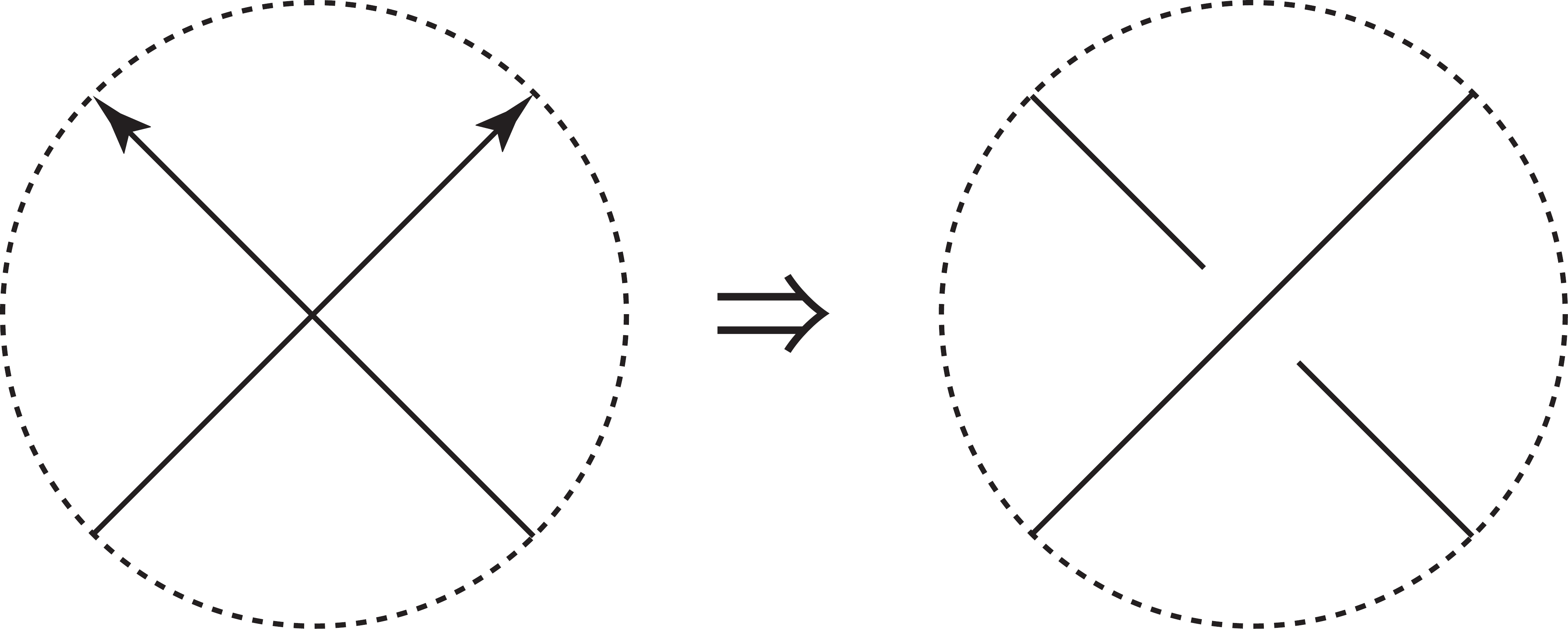}
\caption{A double point to a crossing of a knot diagram}
\label{32f0_1}
\end{center}
\end{figure}
\begin{fact}\label{fact1}
If $P$ and $P'$ are related by a finite sequence consisting of $\rI$ and $\wrIII$, then a knot with knot diagram $K_P$ is isotopic to the knot with knot diagram $K_{P'}$.  As a corollary, a knot projection $P$ and trivial knot projection $O$ are related by a finite sequence consisting of $\rI$ and $\wrIII$ if and only if $P$ and $O$ are related by a finite sequence consisting of $\rI$ \cite[Page ~13, Corollary~4.1]{IT_12}.
\end{fact}
\begin{fact}\label{fact2}
Let $P$ be a knot projection and $O$ be the trivial knot projection.  If $P$ and $O$ are related by a finite sequence consisting of $\rI$ and $\srIII$, then $P$ is a connected sum of knot projections, each of which is $O$, the curve with the shape of $\infty$, or the trefoil projection $3_1$, as shown in Fig.~\ref{32f02} \cite[Page 621, Theorem~4]{ITT}.  
\begin{figure}[htbp]
\begin{center}
\includegraphics[width=6cm]{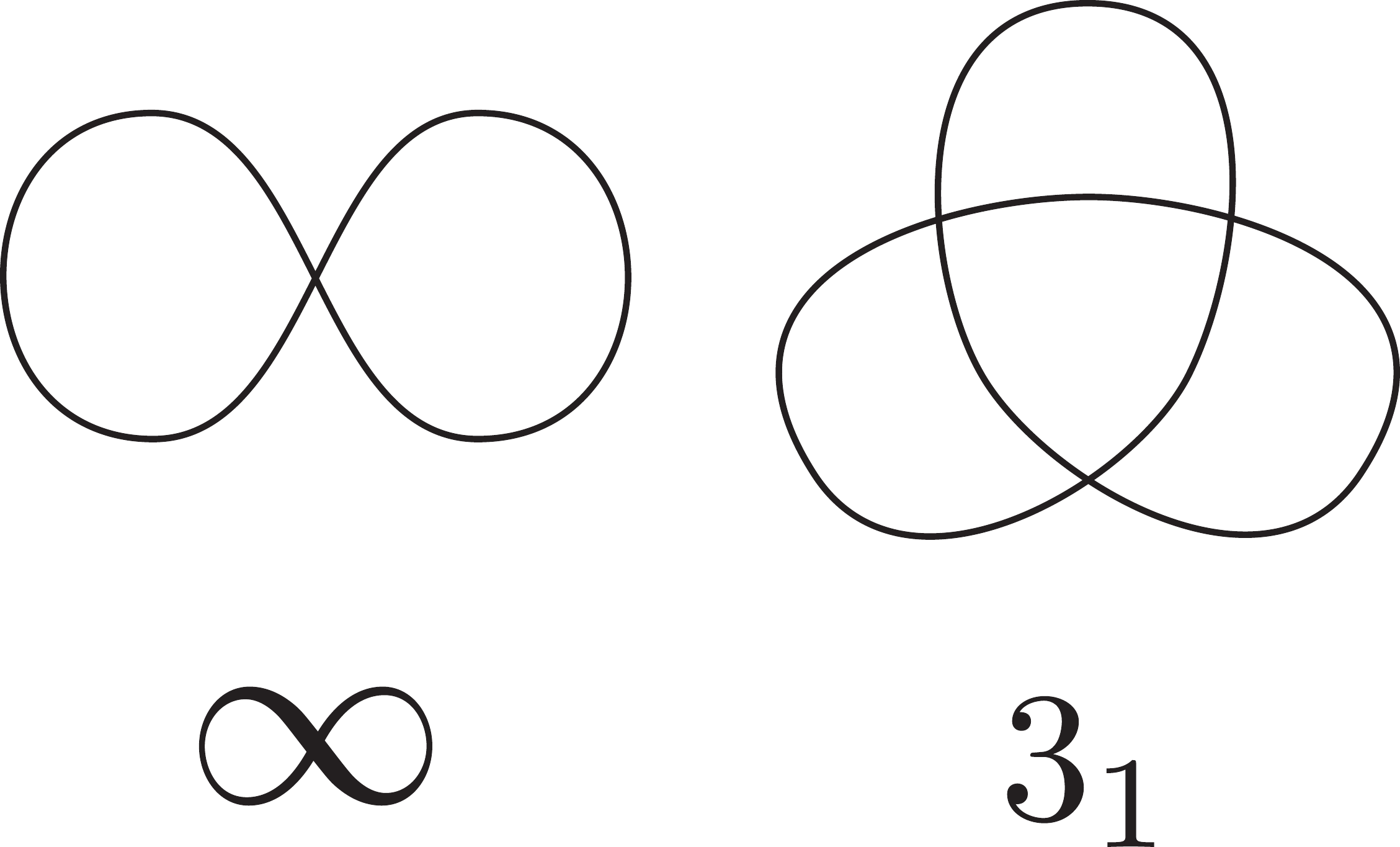}
\caption{The curve with the shape of $\infty$ (left) and trefoil projection $3_1$ (right)}
\label{32f02}
\end{center}
\end{figure}
\end{fact}
For every double point, we locally replace the two branches of a double point with two simple arcs, as shown in Fig.~\ref{32f03}.  
\begin{figure}[htbp]
\begin{center}
\includegraphics[width=5cm]{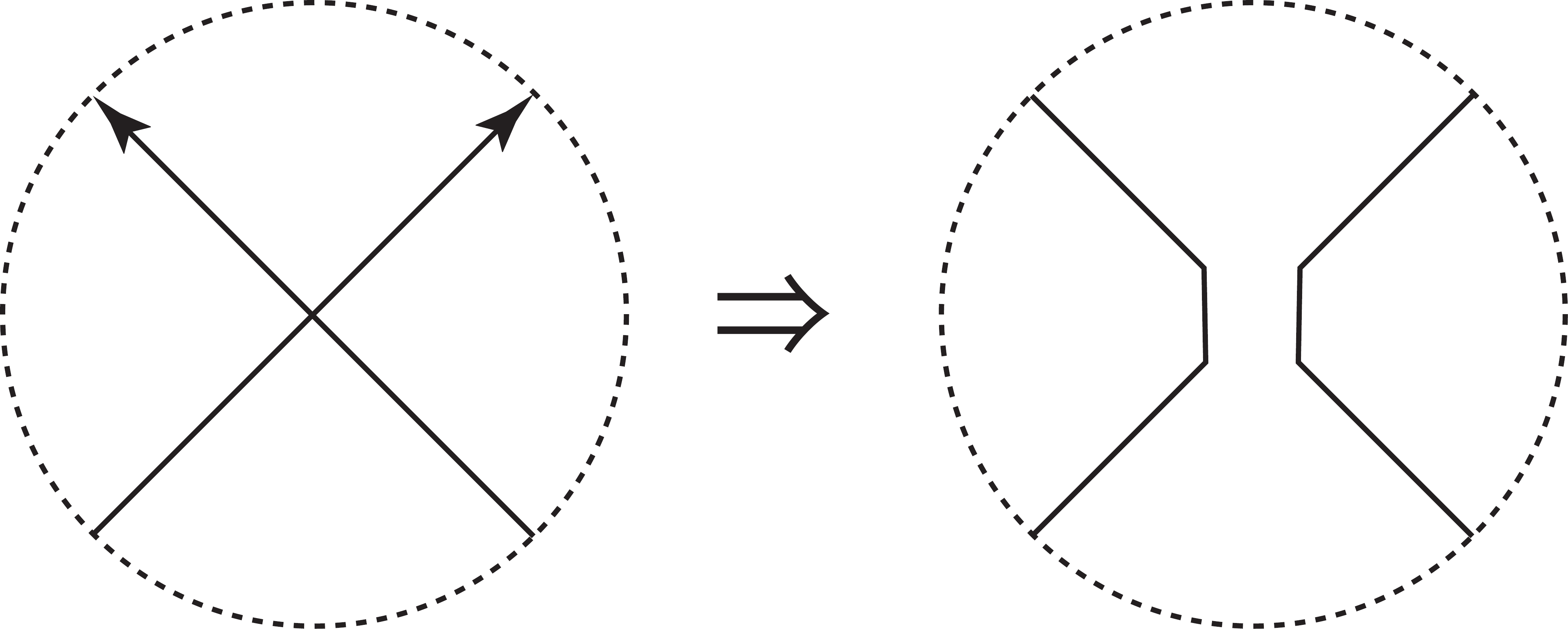}
\caption{Seifert resolution}
\label{32f03}
\end{center}
\end{figure}
After applying the replacements, called {\it{Seifert resolution}}, to all the double points of a knot projection $P$, we obtain an arrangement of circles, called the {\it{Seifert circle arrangement}}, denoted by $S(P)$.  The number of circles in $S(P)$ is called {\it{Seifert circle number}} and is denoted by $s(P)$.     
\begin{fact}[a well-known fact]\label{fact3} Let $P$ be a knot projection.  Then, the Seifert circle arrangement $S(P)$ and the Seifert circle number $s(P)$ are invariant under $\wrII$ and $\wrIII$.  
\end{fact}
\begin{fact}\label{fact4}

\noindent(1) Let $P^r$ be a knot projection obtained from a knot projection $P$ with no $1$-gons and $2$-gons only decreasing double points by a finite sequence consisting of $\rI$ and $\rII$.  Two knot projections, $P^r$ and ${P'}^r$, are sphere isotopic if and only if $P$ and $P'$ are related by a finite sequence consisting of $\rI$ and $\rII$ \cite[Page 2302, Lemma~2.1]{khovanov} (See also \cite[Page~5, Theorem~2.2 (3)]{IT_12} and \cite{IT_Add}).  

\noindent(2) Let $P^{1r}$ (resp.~$P^{2r}$) be a knot projection obtained from a knot projection $P$ with no $1$-gons (resp.~$2$-gons) only decreasing double points by a finite sequence consisting of $\rI$ (resp.~$\rII$).  Two knot projections, $P^{1r}$ and ${P'}^{1r}$ (resp.~$P^{2r}$ and ${P'}^{2r}$), are sphere isotopic if and only if $P$ and $P'$ are related by a finite sequence consisting of $\rI$ (resp.~$\rII$) \cite[Page 5, Theorem~2.2 (1) and (2)]{IT_12}.  
\end{fact}
\begin{fact}\label{fact5}
Let $P^{sr}$ (resp.~$P^{wr}$) be a knot projection obtained from a knot projection $P$ with no $1$-gons and coherent (resp.~incoherent) $2$-gons only decreasing double points by a finite sequence consisting of $\rI$ and $\srII$ (resp.~$\wrII$).  Two knot projections, $P^{sr}$ and ${P'}^{sr}$, are sphere isotopic if and only if $P$ and $P'$ are related by a finite sequence consisting of $\rI$ and $\srII$ (resp.~$\wrII$) \cite[Page 2, Theorem~1]{IT_circle}.    
\end{fact}
\begin{fact}\label{fact_circle}
Let us consider the local replacement of a double point as shown in Fig.~\ref{32f17}.
\begin{figure}[htbp]
\begin{center}
\includegraphics[width=5cm]{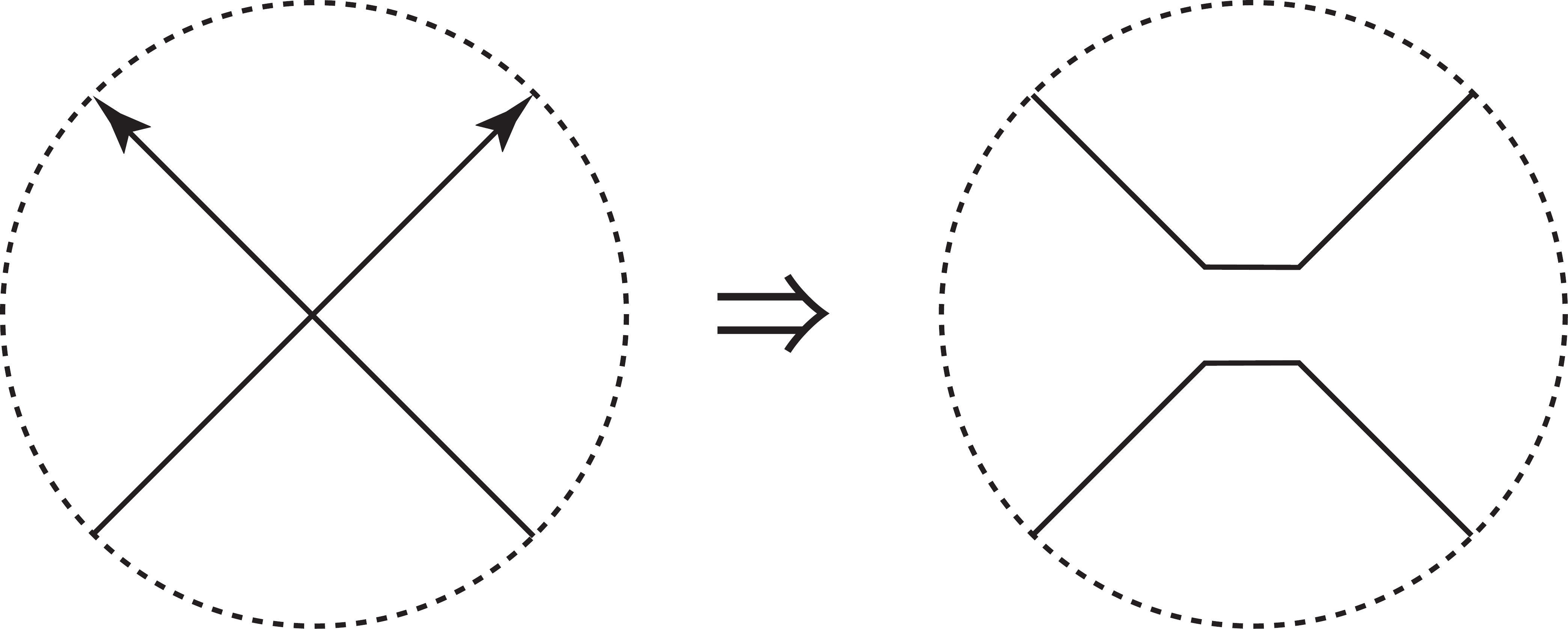}
\caption{Non-Seifert resolution}
\label{32f17}
\end{center}
\end{figure}
After applying the replacements to all the double points of a knot projection $P$, we have an arrangement of circles on $S^{2}$, called {\it{circle arrangement}} $\tau(P)$.  The number of circles in $\tau(P)$ is called the {\it{circle number}} $|\tau(P)|$.  $\tau(P)$ and $|\tau(P)|$ are invariant under $\rI$ and $\srII$ \cite[Page~5, Theorem~3 and Corollary~3]{IT_circle}.  
\end{fact}
\begin{fact}\label{fact_w123}
Let $P$ be a knot projection and $D_P$ a knot diagram obtained by arbitrarily information of any over/under-crossing data.    
Let $tr(P)$ be the trivializing number \cite[Page 440, Theorem~13]{hanaki} and $g(P)$ the canonical genus of $D_P$.  The integer $W(P)=tr(P)-2g(P)$ is invariant under $\rI$, $\wrII$, and $\wrIII$ \cite[Page 4, Theorem~2]{IT_w123}.    
\end{fact}
\begin{fact}\label{fact6}
Let $P$ be a knot projection.  The Arnold invariant ${J^{+}_S}(P)$ is invariant under $\srII$ and $\rIII$ (Polyak mentions that the result was obtained by Arnold, see \cite[Sec.~2.4]{polyak}.  For Arnold's serial work, see \cite{arnold}). 
\end{fact}
\section{Proof of Theorem~\ref{main1}}
\noindent (1) $\mathcal{S}=\emptyset$.
First, we consider the case where $\mathcal{S}$ $=$ $\emptyset$.   It can be easily seen that, for this case, there exists only an equivalence relation, i.e., sphere isotopy.  Here, two knot projections $P$ and $P'$ are called \emph{sphere isotopic} if there exists a smooth family of homeomorphisms $h_t :$$S^2 \to S^2$ for $t \in [0, 1]$ such that $h_0$ is the identity map of $S^2$ and $h_1 (P)=P'$.  Such a family of $h_t$ is called a \emph{sphere isotopy}.   
Thus, the relation $\sim_{\mathcal{S}}$ is different from any other equivalence relation.  

\noindent $\bullet$ ({\bf{Contracting sets.}}) Next, if $\mathcal{S}$ satisfies the condition that every knot projection is equal to the trivial knot projection $O$, $\mathcal{S}$ is called a {\it{contracting set}}.     
Eight contracting sets, each of which satisfies the condition that every knot projection $\sim_{\mathcal{S}}$ $O$, are listed below.  To avoid confusion, we list them by specifying $\mathcal{S}$.  

\noindent(2) $\{\rI, \srII, \wrII, \srIII, \wrIII \}$

\noindent(3) $\{\rI, \srII, \wrII, \srIII \}$

\noindent(4) $\{\rI, \srII, \wrII, \wrIII \}$

\noindent(5) $\{\rI, \srII, \srIII, \wrIII \}$

\noindent(6) $\{\rI, \wrII, \srIII, \wrIII \}$

\noindent(7) $\{\rI, \srII, \srIII \}$ \label{7}

\noindent(8) $\{\rI, \srII, \wrIII \}$ \label{8}

\noindent(9)  $\{\rI, \wrII, \srIII \}$ \label{9}

Recall that every knot projection can be related to the trivial knot projection $O$ by a finite sequence consisting of $\rI$, $\rII$, and $\rIII$.  To show that all the eight sets are contracting sets, we notice that it is sufficient to show that for cases (7), (8), and (9), $\mathcal{S}$ generates $\{ \rI, \rII, \rIII \}$ ($=$ $\{ \rI, \srII, \wrII, \srIII, \wrIII \}$).  
However, the last three cases have been already obtained from \cite[Page 7, Proposition~1]{IT_w123} using Fig.~\ref{32f1}, the proof of which is simple; hence, we describe the proof here again.  
\begin{figure}[h!]
\includegraphics[width=11cm]{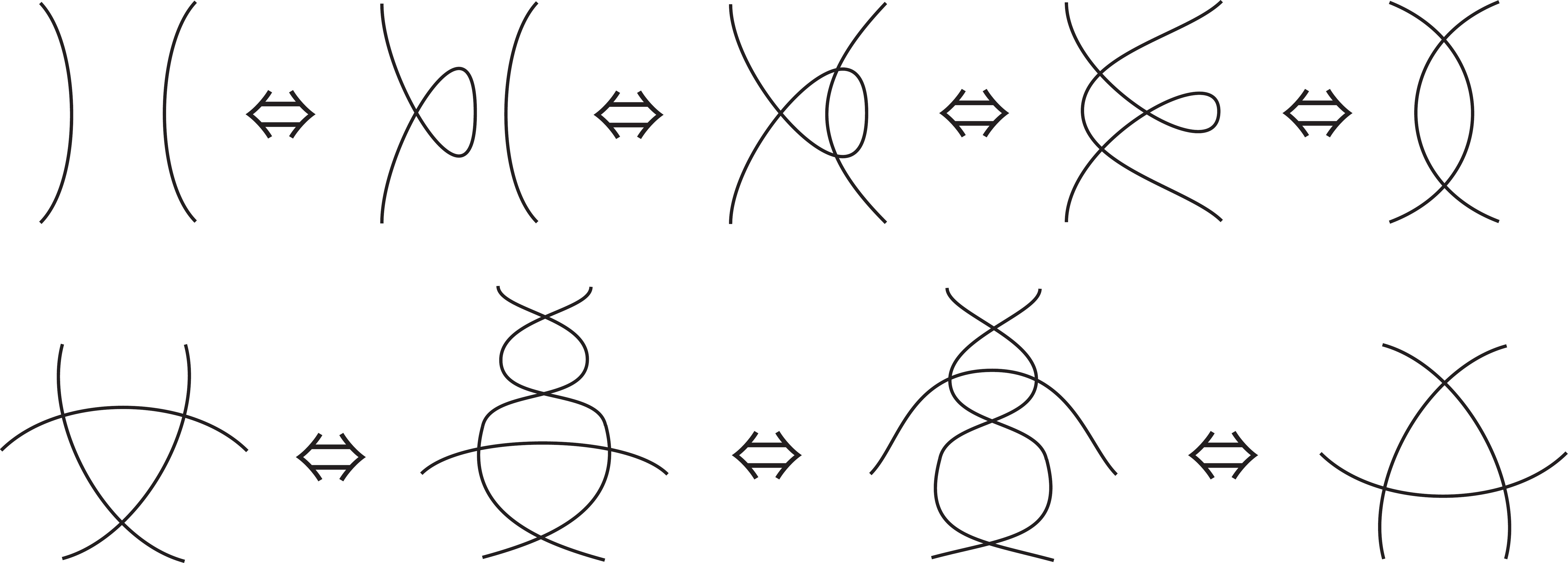}
\caption{Key sequences of figures.
Upper line represents a single $\srII$ (resp.~$\wrII$), which consists of two $\rI{\rm{s}}$, a $\wrII$ (resp.~$\srII$), and a $\srIII$ (resp.~$\wrIII$).  Lower line represents a single $\srIII$ (resp.~$\wrIII$) consists of two ${\srII}{\rm{s}}$ and a $\wrIII$ (resp.~$\srIII$).}\label{32f1}
\end{figure}

\noindent(7) $\{\rI, \srII, \srIII \}$ generates $\{\rI, \rII, \rIII \}$.  This is because $\wrIII$ is generated by $\srII$ and $\srIII$ and $\wrII$ is generated by $\rI$, $\srII$, and $\wrIII$.   

\noindent(8) $\{\rI, \srII, \wrIII \}$ generates $\{ \rI, \rII, \rIII \}$.  This is because $\srIII$ is generated by $\srII$ and $\wrIII$.  The reminder is clear from case (7).

\noindent(9) $\{ \rI, \wrII, \srIII \}$ generates $\{ \rI, \rII, \rIII \}$.  This is because $\srII$ is generated by $\rI$, $\wrII$ and $\srIII$.  The reminder is clear from case (7).

\noindent$\bullet$ ({\bf{23 non-trivial equivalence classes.}})  Only 19 of the 23 cases need to be considered because 4 cases are the same as 2 other cases (see {\bf{Point 1}} and {\bf{Point 2}}).  In the following points, if two sets, $\mathcal{S}_1$ and $\mathcal{S}_2$, of Reidemeister moves generate the same equivalence relation, we say that $\mathcal{S}_1$ is equivalent to $\mathcal{S}_2$ and write $\mathcal{S}_1$ $\sim$ $\mathcal{S}_2$.  

\noindent {\bf{Point 1}}: (10) $\{ \srII, \wrII, \wrIII \}$ 

$\sim$ (11) $\{ \srII, \wrII, \srIII, \wrIII \}$ 

$\sim$ (12) $\{ \srII, \wrII, \srIII \}$ by using Fig.~\ref{32f1}.  

\noindent {\bf{Point 2}}: (13) $\{\srII, \srIII \}$ 

$\sim$ (14) $\{ \srII, \srIII, \wrIII \}$ 

$\sim$ (15) $\{ \srII, \wrIII \}$ by using Fig.~\ref{32f1}.

Now, consider Conditions 1, 2, and 3 to list the remaining 19 cases.  

\noindent$\bullet$ {\bf{Condition 1.}} $\rI \in \mathcal{S}$.  

First, we treat 11 (resp.~8) cases with $\rI \notin \mathcal{S}$ (resp.~$\in \mathcal{S}$), called the {\it{non-RI}} (resp.~{\it{RI}}~) {\it{group}}.  The condition is critical and stated as follows: the trivial knot projection and the curve with the shape of $\infty$ can be related by the elements of $\mathcal{S}$ if and only if $\rI \in \mathcal{S}$.   

\noindent (11) $\{ \srII, \wrII, \srIII, \wrIII \}$

\noindent (14) $\{ \srII, \srIII, \wrIII \}$

\noindent (16) $\{ \wrII, \srIII \}$

\noindent (17) $\{ \wrII, \srIII, \wrIII \}$

\noindent (18) $\{ \srII \}$

\noindent (19) $\{ \wrII \}$

\noindent (20) $\{ \srII, \wrII \}$

\noindent (21) $\{ \wrII, \wrIII \}$

\noindent (22) $\{ \srIII, \wrIII \}$

\noindent (23) $\{ \srIII \}$

\noindent (24) $\{ \wrIII \}$



\noindent$\bullet$ {\bf{Condition 2.}} {\it{An equivalent class containing the trivial knot projection consists of a single element.}}  

A case satisfying Condition 2 is called a {\it{single triviality case}}.  
Cases satisfying (resp.~not satisfying) Condition 2 are (16), (17), (19), (21), (22), (23), and (24) (resp.~(11), (14), (18), and (20)).  This is because $C(P)=0$ $\Leftrightarrow$ $P$ belongs to $[O]$ under (16), (17), (19), (21), (22), (23), or (24) by Theorem~\ref{thm3}.  Thus, we obtain Table~\ref{t1}.  
\begin{table}[h!]
\caption{Non-RI group}\label{t1}
\begin{tabular}{|c|c|} \hline
& non-RI group \\ \hline
non-single triviality & (11), (14), (18), (20)\\ \hline
single triviality & (16), (17), (19), (21), (22), (23), (24) \\ \hline
\end{tabular}
\end{table}

On the other hand, Cases (25)--(32), each of which contains $\rI$, are listed as follows.  These cases are referred to as the RI group.
 
\noindent (25) $\{ \rI \}$

\noindent (26) $\{ \rI, \srII \}$

\noindent (27) $\{ \rI, \wrII \}$

\noindent (28)  $\{ \rI, \srII, \wrII \}$

\noindent (29) $\{ \rI, \srIII \}$

\noindent (30) $\{ \rI, \wrIII \}$

\noindent (31) $\{ \rI, \srIII, \wrIII \}$

\noindent (32) $\{ \rI, \wrII, \wrIII \}$

\noindent$\bullet$ {\bf{Condition 3.}} {\it{There exists an equivalence class containing both the knot projection $P_F$, as shown in Fig.~\ref{32f3}, and the trivial knot projection.}}  

A case satisfying Condition 3 is called {\it{flower trivial}}.
\begin{figure}[htbp]
\begin{center}
\includegraphics[width=2cm]{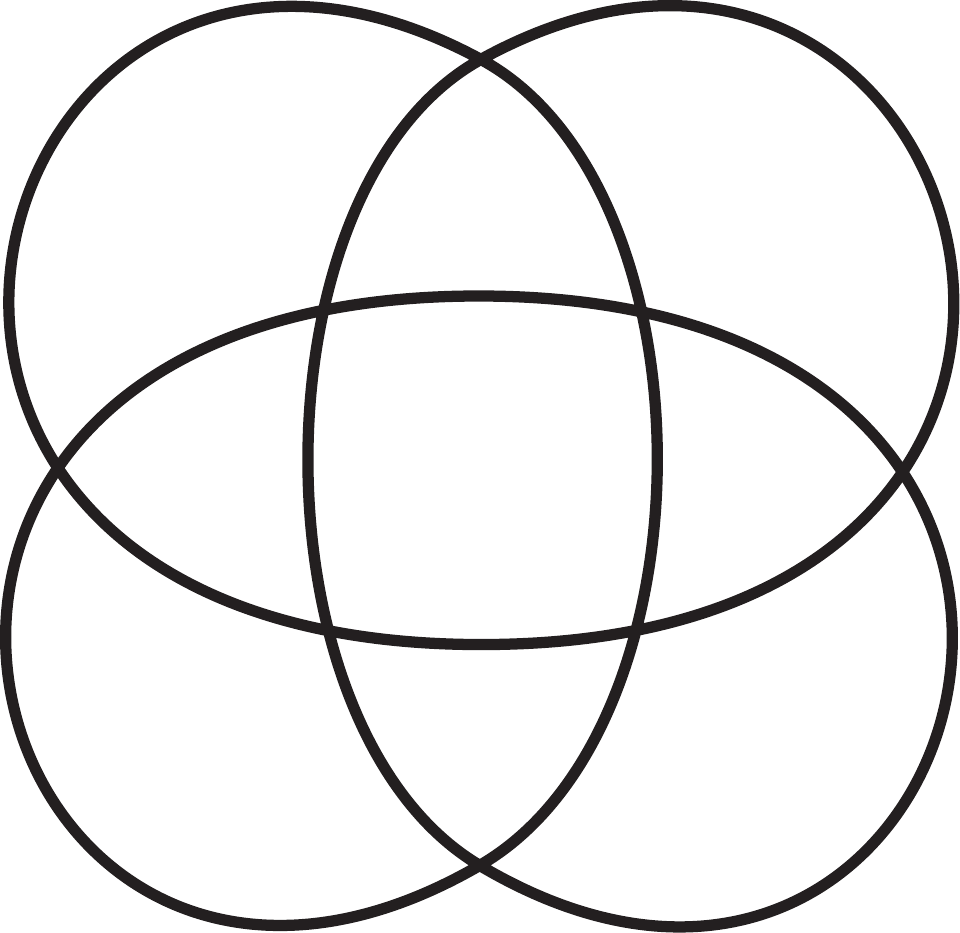}
\caption{$P_F$
}
\label{32f3}
\end{center}
\end{figure}
We now classify the RI group into two subsets as shown in Table~\ref{table2} (note that all cases of the RI group satisfy the non-single triviality) by applying Fact~\ref{fact4}~(2) to (25), Fact~\ref{fact5} to (26), Fact~\ref{fact5} to (27), Fact~\ref{fact4}~(1) to (28), Fact~\ref{fact2} to (29), and Fact~\ref{fact1} to (30).  
\begin{table}[htbp]
\caption{RI group}
\begin{center}
\begin{tabular}{|c|c|} \hline
& RI group (non-single triviality) \\ \hline
non-flower triviality & (25), (26), (27), (28), (29), (30) \\ \hline
flower triviality & (31), (32) \\ \hline
\end{tabular}
\end{center}
\label{table2}
\end{table}

Because it is easier to classify the RI group than the non-RI group, we do that first.  An equivalence class containing $P$ is denoted by $[P]$.  
\subsection{Classification of the non-flower trivial cases in the RI group: (25), (26), (27), (28), (29), and (30)}
The knot projections, $4_1$ and $5_1$ are defined in Fig.~\ref{32f4}.    
\begin{figure}[htbp]
\begin{center}
\includegraphics[width=3cm]{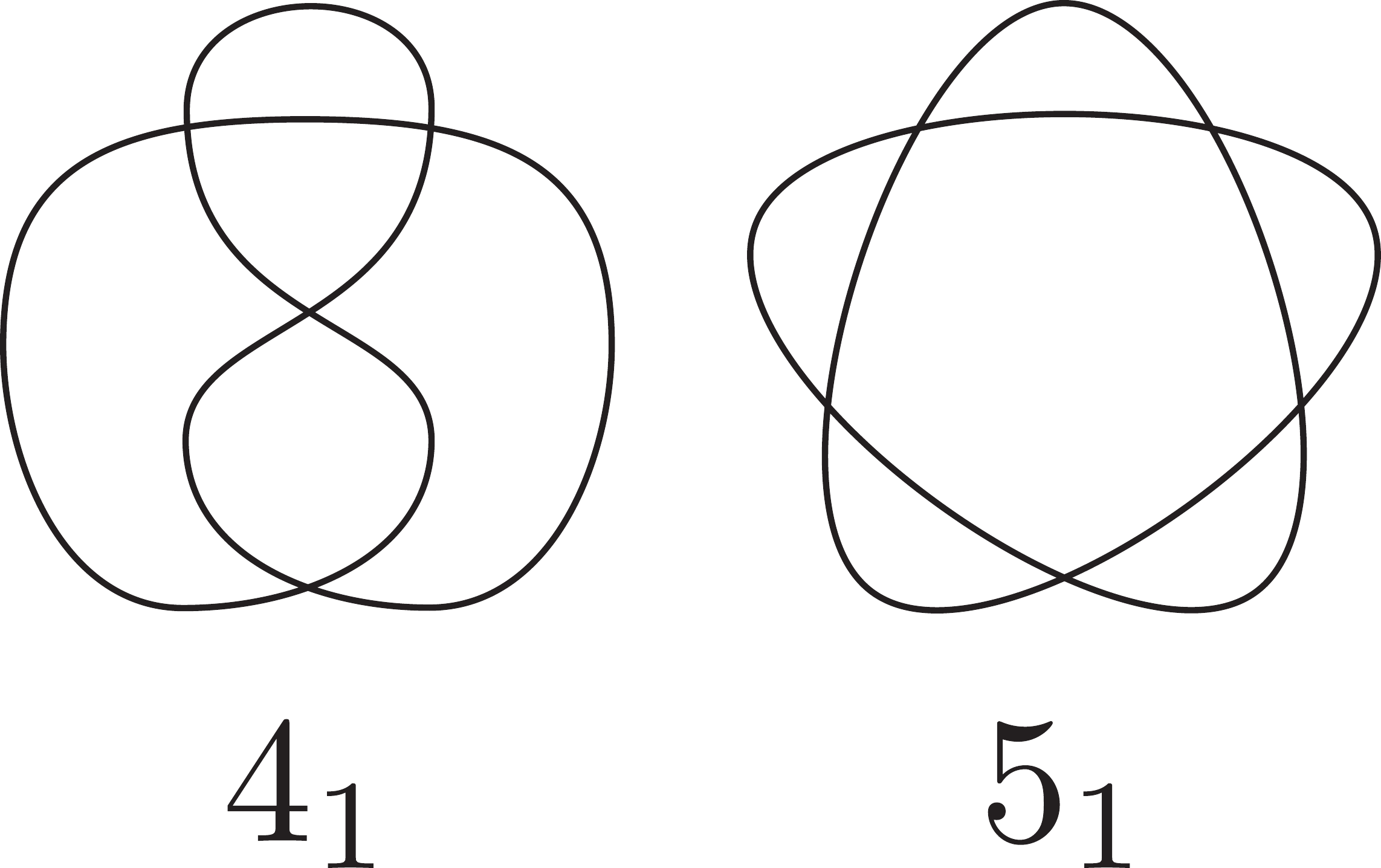}
\caption{$4_1$ and $5_1$}
\label{32f4}
\end{center}
\end{figure}
Table~\ref{table3} classifies the non-flower trivial cases in the RI group except for the pair (25) and (30).  By using the circle number (Fact~\ref{fact_circle}), $|\tau(3_1)|$ $\neq$ $|\tau(4_1)|$ for (25).  However, $[3_1]$ $=$ $[4_1]$ for (30).  
\begin{table}[htbp]
\caption{Non-flower trivial RI group}\label{table3}
\begin{center}
\begin{tabular}{|c|c|c|c|c|}\hline
Case & \multicolumn{3}{|c|}{Formulae} & Key Fact \\ \hline
(25) & $[3_1] \neq [O]$ & $[4_1] \neq [O]$ & $[5_1] \neq [O]$&  Fact~\ref{fact4} (2) \\ \hline
(26) & $[3_1] \neq [O]$ & $[4_1] = [O]$ & $[5_1] \neq [O]$&  Fact~\ref{fact5}  \\ \hline
(27) & $[3_1] = [O]$ & $[4_1] \neq [O]$ & $[5_1] = [O]$&  Fact~\ref{fact5}  \\ \hline
(28) & $[3_1]$ $=$ $[O]$ &$[4_1]$ $=$ $[O]$ & $[5_1]$ $=$ $[O]$&  Fact~\ref{fact4} (1)  \\ \hline
(29) & $[3_1]$ $=$ $[O]$ & $[4_1] \neq [O]$ & $[5_1] \neq [O]$&  Fact~\ref{fact2}  \\ \hline
(30) & $[3_1]$ $\neq$ $[O]$ & $[4_1]$ $\neq$ $[O]$ & $[5_1]$ $\neq$ $[O]$ &  Fact~\ref{fact1}  \\ \hline
\end{tabular}
\end{center}
\end{table}%
\subsection{Classification of the flower trivial cases in the RI group: (31) and (32)}
The knot projection $7_4$ is defined in Fig.~\ref{32f5}.
\begin{figure}[htbp]
\begin{center}
\includegraphics[width=1cm]{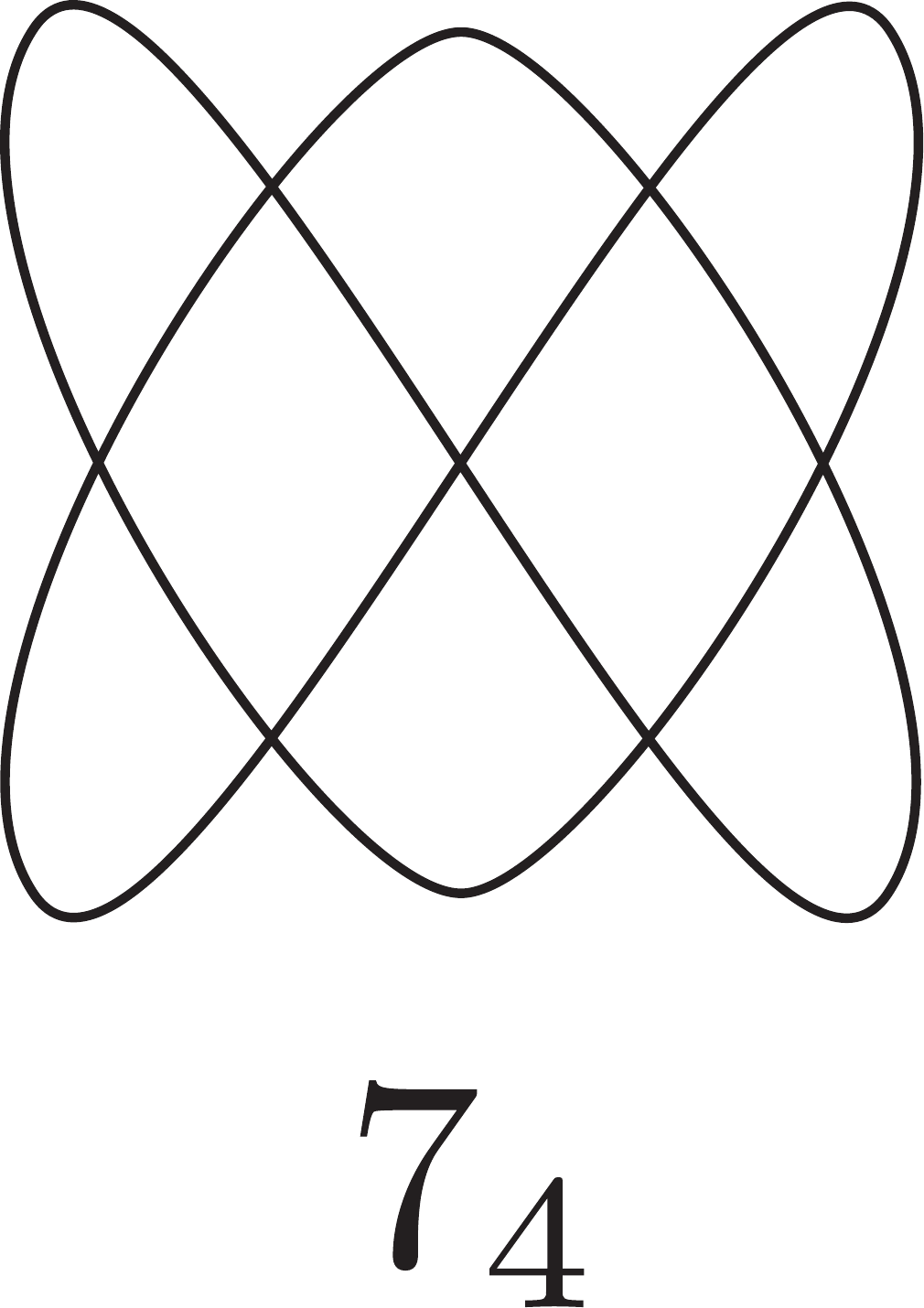}
\caption{$7_4$}
\label{32f5}
\end{center}
\end{figure}
We can see that $[7_4]$ $=$ $[O]$, under the equivalence relation of (31).  However, $[7_4] \neq [O]$, under the equivalence relation of (32) because $W(7_4)$ $=$ $2$ $\neq$ $0$ $=$ $W(O)$ (Fact~\ref{fact_w123}).  This completes the classification of the RI group.  
\subsection{Classification of the single trivial cases in the non-RI group: (16), (17), (19), (21), (22), (23), and (24)}
In this section, the curve appearing as $\infty$ is denoted by $\infty$.     
\begin{table}[htbp]
\caption{Non-RI group satisfying single triviality}
\begin{center}
\begin{tabular}{|c|c|c|c|} \hline
Case & \multicolumn{2}{|c|}{Formulae} & Properties \\ \hline
(16) & $[\infty]$ $=$ $[3_1]$ $=$ $[P_Y]$ & $[P_C] \neq [P_F]$ & $\coh^{\odd}(P_F) \neq \coh^{\odd}(P_C)$ \\
&&& (by Theorem~\ref{thm2}) \\ \hline
(17) & $[\infty]$ $=$ $[3_1]$ $=$ $[P_Y]$ & $[P_C] = [P_F]$ & \\ \hline
(19) & $[\infty]$ $=$ $[3_1]$ $\neq$ $[P_Y]$ & $s(3_1) \neq s(P_Y)$ & $c({P_F}^{2wr})=c(P_F)=8$ \\
& &(by Fact~\ref{fact3}) & ${P_F}^{2wr} = P_F$ (by Theorem~\ref{thm_s2w2}) \\ 
&&& If $[P]=[P_F], c(P) \ge 8$. \\ \hline
(21) & $[\infty]$ $=$ $[3_1]$ $\neq$ $[P_Y]$ & $s(3_1) \neq s(P_Y)$ & $\exists P$ s.t. $[P]=[P_F]$ and $c(P) < 8$ \\
&& (by Fact~\ref{fact3}) & \\ \hline
(22) & $[\infty]$ $\neq$ $[3_1]$ $=$ $[P_Y]$ && $[P_F]$ has at least two elements. \\ \hline
(23) & $[\infty]$ $\neq$ $[3_1]$ $=$ $[P_Y]$ && $[P_F]$ consists of only $P_F$. \\ \hline
(24) & $[\infty]$ $\neq$ $[3_1]$ $\neq$ $[P_Y]$ && $[3_1]$ consists of only $3_1$. \\ \hline
\end{tabular}
\end{center}
\label{table2a}
\end{table}%
First, referring to Table~\ref{table2a}, we can see that  equivalence class $[3_1]$ contains (resp.~does not contain) the curve with the shape of $\infty$ under (16), (17), (19), or (21) (resp.~(22), (23), or (24)) because the number of double points $c(P)$ is invariant under any type of $\rIII$.  Here, let us consider the curve $P_Y$ with no $2$-gons, which is obtained from the trivial knot projection via three $\rI$s increasing double points, as shown in Fig.~\ref{32f8}.  
\begin{figure}[htbp]
\begin{center}
\includegraphics[width=2cm]{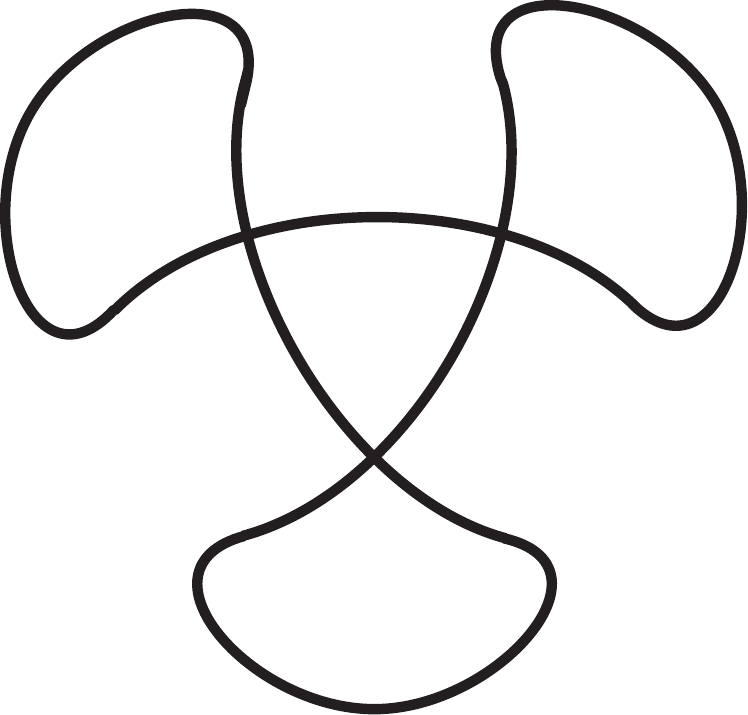}
\caption{$P_Y$}
\label{32f8}
\end{center}
\end{figure}
The curve $P_Y$ is (resp.~is not) an element in $[3_1]$ under (16) or (17) (resp.~(19) or (21)).  
From Fact~\ref{fact3}, $s(P_Y)$ $=$ $4$ $\neq$ $2$ $=$ $s(3_1)$, under the equivalence relation of (19) or (21).  

Second, we distinguish between (16) and (17).  Let $P_C$ be the knot projection defined in Fig.~\ref{32f18}.  
\begin{figure}[htbp]
\begin{center}
\includegraphics[width=2cm]{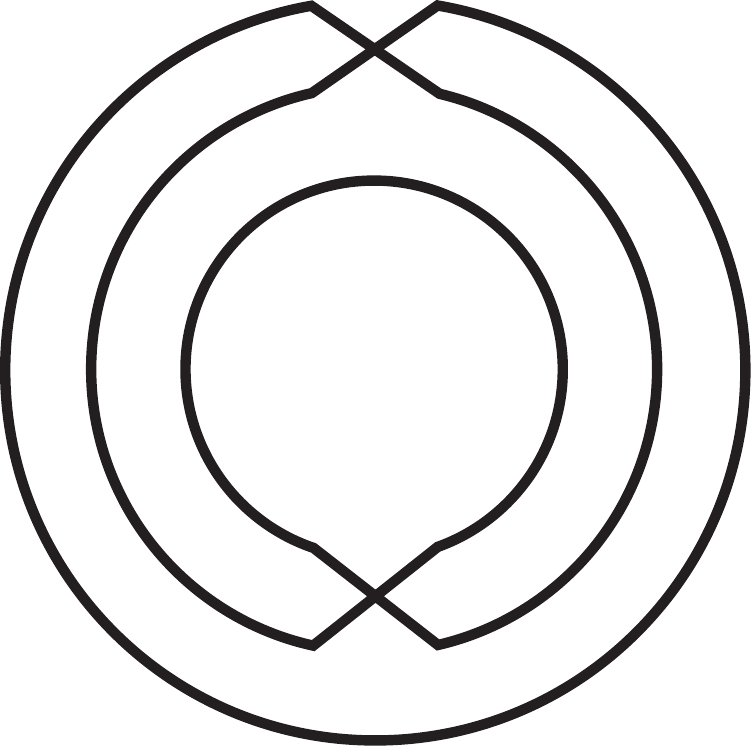}
\caption{$P_C$}
\label{32f18}
\end{center}
\end{figure}
For (16), $\coh^{\odd}(P_F)$ $=$ $0$ and $\coh^{\odd}(P_C)$ $=$ $1$, where $\coh^{\odd}(P)$ is invariant under $\wrII$ and $\srIII$ (i.e., $[P_F]$ $\neq$ $[P_C]$).  On the other hand, $[P_F]$ $=$ $[P_C]$ under the equivalence relation of (17).  

Third, we distinguish between (19) and (21), where (19) and (21) can be detected by the minimum number of double points in class $[P_F]$.  Here, note that from Theorem~\ref{thm_s2w2}, ${P_F}^{2wr}$ $=$ $P_F$ under (19), which implies that the minimum number of double points in the equivalence class is eight. However, $[P_F]$ contains the knot projection with two ($< 8$) double points under (21).

Finally, consider the remaining three equivalence relations (22), (23), and (24).  We notice that $[3_1]$ consists of a single element $3_1$ (resp.~at least two elements) under (24) (resp.~(22) or (23)) because $3_1$ does not have an incoherent (resp.~has a coherent) $3$-gon.  Further, for the flower knot projection $P_F$, $[P_F]$ consists of a single element $P_F$ (resp.~at least have two elements) under (23) (resp.~(22)) because $P_F$ does not have a coherent (resp.~have an incoherent) $3$-gon.  
\subsection{Classification of the non-single trivial cases in the non-RI group: (11), (14), (18), and (20)}
In this section, we say that the symbol $\infty$ indicates the curve with the shape of $\infty$.  Table~\ref{table4} shows the claim.
\begin{table}[htbp]
\caption{Non-RI group satisfying non-single triviality}
\begin{center}
\begin{tabular}{|c|c|c|c|} \hline
Case & Detection & Formulae & Key Fact \\ \hline
(11) & $[\infty] = [3_1] = [P_Y]$ & & \\ \hline
(14) & $[\infty] \neq [3_1] = [P_Y]$ & ${{J_S}^{+}} (\infty)=0 \neq 2 = {{J_S}^{+}} (3_1)$ & Fact~\ref{fact6} \\ \hline
(18) & $[\infty] \neq [3_1] \neq [P_Y]$ & ${3_1}^{2sr} = {3_1}$, ${P_Y}^{2sr} = P_Y$, $\infty^{2sr} = \infty$ & Theorem~\ref{thm_s2w2} \\ \hline
(20) & $[\infty] = [3_1] \neq [P_Y]$ & ${3_1}^{2r}$ $=$ $\infty$, ${P_Y}^{2r} = P_Y$ & Fact~\ref{fact4}~(2)  \\ \hline
\end{tabular}
\end{center}
\label{table4}
\end{table}%
\hfill$\Box$


\end{document}